\documentclass[12pt]{amsart}
\usepackage{amsmath,amssymb}
\usepackage{calrsfs}
\usepackage{url}
\newcommand{\Z}{{\mathbb{Z}}}

\newcommand{\D}{{\mathbb{D}}}
\newcommand{\Q}{{\mathbb{Q}}}

\newcommand{\C}{{\mathbb{C}}}
\newcommand{\F}{{\mathbb{F}}}

\newcommand{\fS}{{\mathfrak{S}}}

\newcommand{\fU}{{\mathfrak{U}}}

\newcommand{\cE}{{\mathcal{E}}}

\newcommand{\cF}{{\mathcal{F}}}

\newcommand{\cH}{{\mathcal{H}}}

\newcommand{\Irr}{{\operatorname{Irr}}}

\renewcommand{\leq}{\leqslant}
\renewcommand{\geq}{\geqslant}

\newtheorem{thm}{Theorem}[section]

\newtheorem{prop}[thm]{Proposition}

\theoremstyle{definition}
\newtheorem{defn}[thm]{Definition}
\newtheorem{exmp}[thm]{Example}
\newtheorem{rema}[thm]{}
\theoremstyle{remark}
\newtheorem{rem}[thm]{Remark}
\numberwithin{equation}{section}

\begin{document}

\title[On the values of unipotent characters]{On the values of 
unipotent characters in bad characteristic}
\author{Meinolf Geck}
\address{IAZ - Lehrstuhl f\"ur Algebra\\Universit\"at Stuttgart\\ 
Pfaffenwaldring 57\\D--70569 Stuttgart\\ Germany}
\email{meinolf.geck@mathematik.uni-stuttgart.de}

\subjclass{Primary 20C33; Secondary 20G40}
\keywords{Finite groups of Lie type, unipotent characters, character 
sheaves} 
\date{November 11, 2017}

\begin{abstract} Let $G(q)$ be a Chevalley group over a finite field
$\F_q$. By Lusztig's and Shoji's work, the problem of computing the 
values of the unipotent characters of $G(q)$ is solved, in principle, by
the theory of character sheaves; one issue in this solution is the 
determination of certain scalars relating two types of class functions 
on $G(q)$. We show that this issue can be reduced to the case where 
$q$ is a prime, which opens the way to use computer algebra methods. 
Here, and in a sequel to this article, we use this approach to solve a 
number of cases in groups of exceptional type which seemed hitherto out of 
reach.
\end{abstract}

\maketitle

\section{Introduction} \label{sec0}

Let $G(q)$ be a group of Lie type over a finite field with $q$ elements. 
This paper is concerned with the problem of computing the values of the 
irreducible characters of $G(q)$. The work of Lusztig \cite{L1}, 
\cite{Lintr} has led to a general program for solving this problem. In 
this framework, one has to establish certain identities of class 
functions on $G(q)$ of the form $R_x=\xi \chi_A$, where $R_x$ denotes 
an ``almost character'' (that is, an explicitly known linear combination 
of irreducible characters) and $\chi_A$ denotes the characteristic function 
of a suitable ``character sheaf'' on the underlying algebraic group~$G$; 
furthermore, $\xi$ is a scalar of absolute value~$1$. This program has been 
successfully carried out in many cases, see, e.g., Bonnaf\'e \cite{Bo3}, 
Lusztig \cite{L3}, \cite{L7}, Shoji \cite{Sclass}, \cite{S7} and 
Waldspurger \cite{wald}, but not in complete generality.

In this paper, we will assume that $G(q)$ is of split type and only
consider the above problem as far as unipotent characters of $G(q)$ are 
concerned, as defined by Deligne and Lusztig \cite{DeLu}. By Shoji's work 
\cite{S2}, \cite{S3}, we know that the desired identities $R_x=\xi \chi_A$ 
as above hold, but the scalars $\xi$ are not yet determined in all cases. 
And even in those cases where they are known, this often required elaborate 
computations. In such cases, the scalars then turn out to behave rather 
uniformly as~$q$ varies (see, e.g., Shoji \cite{Sclass}, \cite{S7}). The 
main theoretical result of this paper, Proposition~\ref{p1}, provides a 
partial, {\it a priori} explanation for this phenomenon; the proof is 
merely an elaboration of ideas which are already contained in Lusztig's 
and Shoji's papers. The fact that we can formulate this result without 
any assumptions on $q$ essentially relies on Lusztig \cite{L10}, where 
the ``cleanness'' of cuspidal character sheaves is established in complete 
generality and, consequently, the principal results of \cite{L2a}--\cite{L2e}
(e.g. \cite[Theorems~23.1 and 25.2]{L2e}) hold unconditionally. 

The main observation of this paper is that the statement of 
Proposition~\ref{p1} can also be exploited in a different way, as 
follows. For a given type of group, we consider the base case where $q=p$ 
is a prime. For a specific value of~$p$, we can use ad hoc methods and/or 
computer algebra systems like {\sf GAP} \cite{gap} to perform all kinds of 
computations within the fixed finite group $G(p)$. If we succeed in this 
way to determine the scalars $\xi$ for $G(p)$, then Proposition~\ref{p1} 
tells us that the analogous result will hold for any power of~$p$. This is 
particularly relevant for ``bad'' primes $p=2,3,5$ which, typically, are 
known to cause additional complications and require separate arguments. We 
illustrate this procedure with a number of examples. In particular, we
determine the scalars $\zeta$ in two cases, where the character table
of $G(\F_p)$ is explicitly known; namely, $F_4$, $E_6$ and $p=2$. For type
$F_4$, our results complete earlier results of Marcelo--Shinoda \cite{MaSh}. 
See also \cite{GeHe} for the discussion of further examples, where the 
complete character table of $G(\F_p)$ is not known.

We assume some familiarity with the character theory of finite groups
of Lie type; see, e.g., \cite{C2}, \cite{first}. The basic reference for the
theory of character sheaves are Lusztig's papers \cite{L2a}--\cite{L2e}. In 
Section~\ref{sec1}, we review the classification of unipotent characters 
of $G(q)$ and the analogous classification of the unipotent character 
sheaves on $G$. These two classifications are known to be the same for 
$G(q)$ of split type (a fact which has only recently found a conceptual 
explanation; see Lusztig \cite{L11}). In Section~\ref{sec2}, we can then 
formulate in precise terms the problem of equating class functions 
$R_x=\xi\chi_A$ as above, and establish Proposition~\ref{p1}. Finally, 
Sections~\ref{sec3}, \ref{secF4}, \ref{secE6} contain a number of examples 
where we show how the scalars $\xi$ can be determined using standard functions 
and algorithms in {\sf GAP}.

\begin{rema} {\bf Notation.} \label{abs11} Let $\ell$ be a prime such that
$\ell\nmid q$. If $\Gamma$ is a finite group, we denote by $\mbox{CF}
(\Gamma)$ the vector space of $\overline{\Q}_\ell$-valued functions 
on $G^F$ which are constant on the conjugacy classes of $\Gamma$. (We take 
$\overline{\Q}_\ell$ instead of $\C$ since, in the framework of \cite{DeLu}, 
\cite{Lintr}, class functions on $\Gamma=G(q)$ are constructed whose values 
are cyclotomic numbers in $\overline{\Q}_\ell$.) Given $f,f'\in \mbox{CF}
(\Gamma)$, we denote by $\langle f,f' \rangle=|\Gamma|^{-1} \sum_{g \in 
\Gamma} f(g)\overline{f'(g)}$ the standard scalar product of $f,f'$ where 
the bar denotes a field automorphism which maps roots of unity to their 
inverses. Let $\Irr(\Gamma)$ be the set of irreducible characters of 
$\Gamma$ over $\overline{\Q}_\ell$; these form an orthonormal basis of 
$\mbox{CF}(\Gamma)$ with respect to the above scalar product. 
\end{rema}

\section{Unipotent character sheaves and almost characters} \label{sec1}

Let $p$ be a prime and $k=\overline{\F}_p$ be an algebraic 
closure of the field with $p$ elements. Let $G$ be a connected reductive 
algebraic group over $k$ and assume that $G$ is defined over the 
finite subfield $\F_q\subseteq k$, where $q=p^f$ for some $f\geq 1$. Let 
$F\colon G\rightarrow G$ be the corresponding Frobenius map. Let 
$B\subseteq G$ be an $F$-stable Borel subgroup and $T\subseteq B$ be an 
$F$-stable maximal torus. Let $W=N_G(T)/T$ be the corresponding Weyl group. 
We assume that $F$ acts trivially on $W$ and that $F(t)=t^q$ for all 
$t\in T$. Then the group of rational points $G^F=G(\F_q)$ is a finite group 
of Lie type of ``split type''. 

\begin{rema} \label{r11} For each $w\in W$, let $R_w$ be the virtual 
character of $G^F$ defined by Deligne--Lusztig \cite[\S 1]{DeLu}. Let 
$\fU(G^F)$ be the set of unipotent characters of $G^F$, that is,
\[ \fU(G^F)=\{\rho\in\Irr(G^F)\mid \langle \rho,R_w\rangle\neq 0\mbox{ for 
some $w\in W$}\}.\]
Now \cite[Main Theorem 4.23]{L1} provides a classification of $\fU(G^F)$ in 
terms of 
\begin{itemize}
\item a parameter set $X(W)$ and a pairing $\{\;,\;\} \colon X(W)\times 
X(W)\rightarrow \overline{\Q}_\ell$ (which only depend on $W$),
\item an embedding $\Irr(W) \hookrightarrow X(W)$, $\epsilon\mapsto 
x_\epsilon$.
\end{itemize}
(See \cite[4.21]{L1} for precise definitions; recall that we assume
that $F$ acts trivially on $W$). Indeed, there is a bijection 
\[ \fU(G^F)\leftrightarrow X(W),\qquad \rho\leftrightarrow x_\rho,\]
such that for any $\rho\in\fU(G^F)$ and any $\epsilon\in\Irr(W)$, we have
\[ \langle \rho,R_\epsilon\rangle=\delta_\rho\{x_\rho,x_\epsilon\}\qquad
\mbox{where} \qquad R_\epsilon:=\frac{1}{|W|} \sum_{w\in W} \epsilon(w)R_w.\]
Here, for $\rho\in\fU(G^F)$, we define a sign $\delta_\rho=\pm 1$ by the
condition that $\delta_\rho D_G(\rho)\in\fU(G^F)$, where $D_G$ denotes 
the duality operator on the character ring of $G$; see \cite[6.8]{L1}.
Note that \cite[6.20]{L1} identifies $\delta_\rho$ with the 
sign $\Delta(x_\rho)$ appearing in the formulation of \cite[4.23]{L1}.
\end{rema}

\begin{rem} \label{r11a} Assume that $G$ is simple modulo its centre. The 
bijection $\fU(G^F)\leftrightarrow X(W)$ in \ref{r11} is not uniquely 
determined by the properties stated above. We shall make a definite choice 
according to \cite[12.6]{L1} and the tables in the appendix of \cite{L1}. 
In particular, this means the following. Let us fix a square root of $q$ in 
$\overline{\Q}_\ell$. Then it is well-known that the irreducible characters 
of $G^F$ which occur in the character of the permutation representation of 
$G^F$ on the cosets of $B^F$ are naturally parametrised by the irreducible 
characters of~$W$; see, e.g., \cite[8.7]{L1}. If $\epsilon\in\Irr(W)$, we 
denote the corresponding irreducible character of $G^F$ by~$\rho_\epsilon$; 
clearly, $\rho_\epsilon \in \fU(G^F)$. Hence, under a bijection $\fU(G^F)
\leftrightarrow X(W)$ as in \ref{r11}, the character $\rho_\epsilon$ will 
correspond to an element $x_{\rho_\epsilon}$. By \cite[Prop.~12.6]{L1}, we 
automatically have $x_\epsilon=x_{\rho_\epsilon}$ except when $\epsilon(1)=
512$ and $G$ is of type $E_7$, or when $\epsilon(1)=4096$ and $G$ is of 
type $E_8$. In these exceptional cases, $\fU(G^F)\leftrightarrow X(W)$ 
can still be chosen such that $x_\epsilon=x_{\rho_\epsilon}$; see the tables 
for $E_7$, $E_8$ in the appendix of \cite{L1}. 

In order to obtain a full uniqueness statement, one has to take into account 
Harish-Chandra series and further invariants of the characters in $\fU(G^F)$, 
namely, the ``eigenvalues of Frobenius'' as determined in \cite[11.2]{L1}; 
see \cite[Prop.~6.4]{DiMi1}, \cite[\S 3]{L10a}, \cite[\S 4]{first}.
\end{rem}

\begin{rema} \label{r12} The ``Fourier matrix'' $\Upsilon:=
\bigl(\{x,y\} \bigr)_{x,y\in X(W)}$ is hermitian, and $\Upsilon^2$ is the 
identity matrix (see \cite[4.14]{L1}). For each $x\in X(W)$, the 
corresponding unipotent ``almost character'' $R_x$ is defined by 
\[ R_x:=\sum_{\rho\in\fU(G^F)} \delta_{\rho} \{x_\rho,x\}\rho;\qquad
\mbox{see \cite[4.24.1]{L1}}.\]
Note that $R_{x_\epsilon}=R_\epsilon$ for $\epsilon\in\Irr(W)$. 
For any $x,y \in X(W)$ we have 
\[ \langle R_x,R_y\rangle=\left\{\begin{array}{cl} 1 & \qquad \mbox{if
$x=y$}, \\ 0 & \qquad \mbox{otherwise}.\end{array}\right.\]
Since $\Upsilon^2$ is the identity matrix, we obtain 
\[ \rho=\delta_{\rho}\sum_{x\in X(W)} \{x,x_\rho\}R_x \qquad \mbox{for any
$\rho\in\fU(G^F)$}.\]
Thus, the problem of computing the values of $\rho\in\fU(G^F)$ is 
equivalent to the analogous problem for the unipotent almost characters 
$R_x$, $x\in X(W)$.
\end{rema}

\begin{rema} \label{r13} Let $\hat{G}$ be the set of character sheaves on
$G$ (up to isomorphism). These are certain simple perverse sheaves in the 
bounded derived category $\mathcal{D}G$ of constructible 
$\overline{\Q}_\ell$-sheaves on $G$ (in the sense of Beilinson, Bernstein, 
Deligne \cite{bbd}), which are equivariant for the action of $G$ on itself 
by conjugation. For $w\in W$ let $K_w^{\mathcal{L}_0}\in\mathcal{D}G$ 
be defined as in \cite[2.4]{L2a}, where $\mathcal{L}_0=\overline{\Q}_\ell$ is 
the constant local system on the maximal torus~$T$. Let $\hat{G}^{\text{un}}$ 
be the set of unipotent character sheaves, that is, those $A\in \hat{G}$ 
which are constituents of a perverse cohomology sheaf ${^p\!H}^i
(K_w^{\mathcal{L}_0})$ for some $w\in W$ and some $i\in\Z$ (see 
\cite[Def.~2.10]{L2a}). Let $\epsilon \in\Irr(W)$. In analogy to the above
definition of $R_\epsilon$, we formally define
\[K_\epsilon^{\mathcal{L}_0}:=\frac{1}{|W|} \sum_{w\in W} \epsilon(w)
\sum_{i\in\Z} (-1)^{i+\dim G}\, {^p\!H}^i(K_w^{\mathcal{L}_0});\] 
see \cite[14.10.3]{L2c}. (We write $K_\epsilon^{\mathcal{L}_0}$ in order 
to avoid confusion with $R_\epsilon$ in \ref{r11}.) As in 
\cite[14.10.4]{L2c}, we also denote by $(A:K_\epsilon^{\mathcal{L}_0})$ 
the multiplicity of $A\in\hat{G}^{\text{un}}$ in $K_\epsilon^{\mathcal{L}_0}$
(in the appropriate Grothendieck group). 
\end{rema}

\begin{rema} \label{r14} Now \cite[Theorem 23.1]{L2e} (see also the comments
in \cite[3.10]{L10}) provides a classification of $\hat{G}^{\text{un}}$ in 
terms of similar ingredients as in \ref{r11}. Indeed, let the parameter set
$X(W)$, the pairing $\{\;,\;\} \colon X(W)\times X(W)\rightarrow 
\overline{\Q}_\ell$ and the embedding $\Irr(W) \hookrightarrow X(W)$ be as 
above. Then there is a bijection 
\[ \hat{G}^{\text{un}}\leftrightarrow X(W),\qquad A\leftrightarrow x_A,\]
such that $(A:K_\epsilon^{\mathcal{L}_0}) =\hat{\varepsilon}_{A}\{x_A,
x_\epsilon\}$ for any $A\in \hat{G}^{\text{un}}$ and $\epsilon\in\Irr(W)$. 
Here, we set 
\[\hat{\varepsilon}_K:=(-1)^{\dim G-\dim\text{supp}(K)}\qquad \mbox{for any
$K\in\mathcal{D}G$},\]
where $\mbox{supp}(K)$ is the Zariski closure of the set 
$\{g\in G\mid \cH_g^i(K)\neq \{0\} \mbox{ for some $i$}\}$. (Cf.\ 
\cite[15.11]{L2c}.) Here, $\cH_g^i(K)$ are the stalks at~$g\in G$ of 
the cohomology sheaves of~$K$, for any $i\in \Z$.

Assume that $G$ is simple modulo its centre. Then, again, the bijection 
$\hat{G}^{\text{un}}\leftrightarrow X(W)$ is not uniquely determined by 
the above properties. But one obtains a full uniqueness statement by an 
analogous scheme as in Remark~\ref{r11a}; see \cite[\S 3]{L10a}. 
\end{rema}

\begin{rema} \label{r15} Consider any object $A\in \mathcal{D}G$ and 
suppose that its inverse image $F^*A$ under the Frobenius map is isomorphic
to $A$ in $\mathcal{D}G$. Let $\phi \colon F^*A \stackrel{\sim}{\rightarrow} 
A$ be an isomorphism. Then $\phi$ induces a linear map $\phi_{i,g}\colon 
\cH_g^i(A)\rightarrow \cH_g^i(A)$ for each $i$ and $g\in G^F$. This 
gives rise to a class function $\chi_{A,\phi}\in \mbox{CF}(G^F)$, called 
``characteristic function'' of $A$, defined by 
\[\chi_{A,\phi}(g)=\sum_i (-1)^i \mbox{Trace}(\phi_{i,g},\cH_g^i(A))\qquad
\mbox{for $g \in G^F$},\] 
see \cite[8.4]{L2b}. Note that $\phi$ is unique up to a non-zero scalar;
hence, $\chi_{A,\phi}$ is unique up to a non-zero scalar.

Now assume that $A\in\hat{G}$. Then one can choose an isomorphism $\phi_A
\colon F^*A \stackrel{\sim}{\rightarrow} A$ such that the values of 
$\chi_{A,\phi_A}$ are cyclotomic integers and $\langle \chi_{A,\phi_A},
\chi_{A,\phi_A}\rangle=1$; see \cite[25.6, 25.7]{L2e} (and also the comments
in \cite[3.10]{L10}). The precise conditions which guarantee these 
properties are formulated in \cite[13.8]{L2c}, \cite[25.1]{L2e}; note that 
these conditions specify $\phi_A$ up to multiplication by a root of unity.
In the following, we will tacitly assume that $\phi_A$ has been chosen
in this way whenever $A\cong F^*A$. 
\end{rema}

\begin{thm}[Shoji \protect{\cite[5.7]{S2}, \cite[3.2, 4.1]{S3}}] \label{r16} 
Assume that $Z(G)$ is connected and that $G/Z(G)$ is simple; also recall that 
$F$ is assumed to act trivially on $W$. Let $A\in\hat{G}^{\operatorname{un}}$ 
and $x\in X(W)$ be such that $x=x_A$. Then $F^*A\cong A$ and 
$R_x$ is equal to $\chi_{A,\phi_A}$, up to a non-zero scalar multiple.
\end{thm}

As already mentioned, Shoji's results also apply to non-split groups and to 
non-unipotent characters. In \cite{S2}, \cite{S3} it is assumed, however, 
that $p$ is ``almost good'', that is, the following conditions hold. If $G$ 
is type $A_n$, $B_n$, $C_n$ or $D_n$, no condition. If $G$ is of type $E_6$, 
then $p\neq 2$. If $G$ is of type $G_2$, $F_4$ or $E_7$, then $p\neq 2,3$. 
If $G$ is of type $E_8$, then $p\neq 2,3,5$ (see \cite[23.0.1]{L2e}). 
Since Lusztig \cite{L10} has established the ``cleanness'' of cuspidal 
character sheaves in full generality, the methods in \cite{S2}, \cite{S3}
which were used to prove Theorem~\ref{r16} for $G$ of exceptional type and 
almost good~$p$ can now also be applied for any~$p$. Consequently, as 
Shoji pointed out to the author, Theorem~\ref{r16} holds as stated above,
without any condition on~$p$. 

\begin{defn}  \label{s1} In the setting of Theorem~\ref{r16}, let $A\in 
\hat{G}^{\operatorname{un}}$ and $x\in X(W)$ be such that $x=x_A$. Recall 
that $\phi_A\colon F^*A \stackrel{\sim}{\rightarrow} A$ is assumed to be
chosen as in \ref{r15}. Then we define $0\neq \zeta_A\in\overline{\Q}_\ell$
by the condition that $R_x=(-1)^{\dim G}\hat{\varepsilon}_A\zeta_A\chi_{A,
\phi_A}$.
\end{defn}

\begin{rema} \label{r17} In the setting of \ref{r11a}, let $\epsilon\in
\Irr(W)$ and consider the corresponding character~$\rho_\epsilon$. We have 
$\delta_{\rho_\epsilon}=1$ and $x_\epsilon=x_{\rho_\epsilon}$. So, using
the formula in \ref{r12}, we obtain:
\begin{align*} 
\rho_\epsilon&=\sum_{x\in X(W)} \{x, x_{\rho_\epsilon}\}R_x=
\sum_{x\in X(W)} \{x, x_\epsilon\}R_x\\&=(-1)^{\dim G}\sum_{A\in 
\hat{G}^{\text{un}}} \hat{\varepsilon}_A\zeta_A\{x_A, x_\epsilon\}
\chi_{A,\phi_A}\qquad \mbox{(see Def.~\ref{s1})}\\&=(-1)^{\dim G}
\sum_{A\in \hat{G}^{\text{un}}} \zeta_A(A:K_\epsilon^{\mathcal{L}_0})
\chi_{A,\phi_A} \qquad \mbox{(see \ref{r14})}.
\end{align*}
Such an expression for $\rho_\epsilon$ as a linear combination of 
characteristic functions first appeared in \cite[14.14]{L2c}; it is actually 
an important ingredient in the proof of Theorem~\ref{r16}. On the other hand,
the argument in \cite[14.14]{L2c} relies on an alternative interpretation 
of the coefficients $\zeta_A$, which we will consider in more detail in 
the following section. Note that, for a given $A\in\hat{G}^{\text{un}}$, 
there always exists some $\epsilon\in\Irr(W)$ such that 
$(A:K_\epsilon^{\mathcal{L}_0})\neq 0$; see \cite[14.12]{L2c}. Furthermore, 
by \cite[25.2]{L2e} (and the comments in \cite[3.10]{L10}), the functions 
$\{\chi_{A,\phi_A}\mid A\in \hat{G}^{\text{un}}\}$ are linearly independent. 
It follows that the coefficients $\zeta_A$ are uniquely determined by the 
above system of equations, where $\epsilon$ runs over $\Irr(W)$. 
\end{rema}

\section{The scalars $\zeta_A$} \label{sec2}

We keep the basic assumptions of the previous section; we also assume
that $Z(G)$ is connected and $G/Z(G)$ is simple. We fix a square root of~$q$ 
in $\overline{\Q}_\ell$. For any $A\in\hat{G}^{\text{un}}$, we know by
Theorem~\ref{r16} that $F^*A\cong A$; we assume that an isomorphism 
$\phi_A\colon F^*A \stackrel{\sim}{\rightarrow} A$ has been chosen as 
in~\ref{r15}. Our aim is to get hold of the coefficients $\zeta_A$ in 
Definition~\ref{s1}. 

\begin{rema} \label{r18} For the further discussion, it will be convenient
to change the notation and label everything by elements of $X(W)$. Thus,
via the bijection $\fU(G^F)\leftrightarrow X(W)$ in \ref{r11} (arranged as 
in Remark~\ref{r11a}), we can write
\[ \fU(G^F)=\{\rho_x \mid x \in X(W)\} \qquad \mbox{where}\qquad
\rho_\epsilon=\rho_{x_\epsilon} \mbox{ for $\epsilon \in \Irr(W)$}.\]
For $x \in X(W)$, we write $\delta_x:=\delta_{\rho_x}$. Then $R_y=\sum_{x 
\in X(W)}\delta_x \{x,y\}\rho_x$ for all $y \in X(W)$. Next, via the
bijection $\hat{G}^{\text{un}}\leftrightarrow X(W)$ in \ref{r14}, we can 
write
\[ \hat{G}^{\text{un}}=\{A_x \mid x \in X(W)\}.\]
For $x \in X(W)$, we denote an isomorphism $F^*A_x \cong A_x$ as 
in \ref{r15} by $\phi_x$ and the corresponding characteristic function
simply by~$\chi_x$. Then the relation in Definition~\ref{s1} is rephrased as 
\[ R_x=(-1)^{\dim G} \hat{\varepsilon}_x \zeta_x\chi_x \qquad \mbox{where}
\qquad \hat{\varepsilon}_x:=\hat{\varepsilon}_{A_x} \mbox{ and } \zeta_x:=
\zeta_{A_x}.\]
For $\epsilon\in\Irr(W)$, the identity in~\ref{r17} now reads: 
\[ \rho_{\epsilon}=(-1)^{\dim G} \sum_{x\in X(W)} \zeta_x(A_x:
K_\epsilon^{\mathcal{L}_0}) \chi_x.\]
\end{rema}

\begin{rema} \label{r21} Let us fix an integer $m\geq 1$. Then $G$ is also 
defined over $\F_{q^m}$ and $F^m\colon G\rightarrow G$ is the corresponding 
Frobenius map. Clearly, $F^m$ acts trivially on $W$ and we have $F^m(t)=
t^{q^m}$ for all $t\in T$. So the whole discussion in Section~\ref{sec1} can 
be applied to $F^m$ instead of~$F$. As in \ref{r18}, we write 
\[ \fU(G^{F^m})=\{\rho_x^{(m)} \mid x \in X(W)\}.\]
Again, the unipotent characters of $G^{F^m}$ which occur in the character 
of the permutation representation of $G^{F^m}$ on the cosets of $B^{F^m}$ 
are naturally parametrised by $\Irr(W)$. If $\epsilon\in\Irr(W)$, we denote 
the corresponding character of $G^{F^m}$ by~$\rho_\epsilon^{(m)}$. As in 
Remark~\ref{r11a}, the labelling of $\fU(G^{F^m})$ is arranged such that 
\[\rho_\epsilon^{(m)}=\rho_{x_\epsilon}^{(m)} \qquad \mbox{for $\epsilon \in 
\Irr(W)$}.\]
For $y \in X(W)$, the corresponding unipotent almost character of $G^{F^m}$ 
is given by
\[ R_y^{(m)}=\sum_{x \in X(W)} \delta_x\{x,y\}\rho_x^{(m)}.\]
Note that $\delta_{\rho_x^{(m)}}=\delta_{\rho_x}=\delta_x$ by
\cite[4.23, 6.20]{L1}.
\end{rema}

\begin{rema} \label{r22} Let $x\in X(W)$. Then $\phi_x\colon F^*A_x
\stackrel{\sim}{\rightarrow} A_x$ naturally induces isomorphisms
\[ F^*(\phi_x)\colon (F^*)^2A_x\stackrel{\sim}{\rightarrow} F^*A_x,\quad
(F^*)^2(\phi_x)\colon (F^*)^3A_x\stackrel{\sim}{\rightarrow} (F^*)^2A_x,
\quad \ldots,\]
which give rise to an isomorphism 
\[\tilde{\phi}_x^{(m)}:=\phi_x\circ F^*(\phi_x)\circ \ldots \circ 
(F^*)^{m-1}(\phi_x) \colon (F^*)^mA_x\stackrel{\sim}{\rightarrow} A_x.\] 
We also have a canonical isomorphism $(F^*)^mA_x\cong (F^m)^* A_x$ which,
finally, induces an isomorphism
\[\phi_x^{(m)} \colon (F^m)^*A_x\stackrel{\sim}{\rightarrow} A_x\qquad
\mbox{(see \cite[1.1]{S2})}.\]
The latter isomorphism again satisfies the 
conditions in \ref{r15}; we denote the corresponding characteristic function
by $\chi_x^{(m)}\colon G^{F^m}\rightarrow \overline{\Q}_\ell$. Note that,
if $g$ is an element in $G^F$ (and not just in $G^{F^m}$), then 
\[ \chi_x^{(m)}(g)=\sum_i (-1)^i\mbox{Trace}\bigl((\phi_x)_{i,g}^m,
\cH_g^i(A_x)\bigr).\]
(See again \cite[1.1]{S2}.) As in Definition~\ref{s1}, we define $0\neq 
\zeta_x^{(m)}\in \overline{\Q}_\ell$ by the condition that
\[ R_x^{(m)}=(-1)^{\dim G}\hat{\varepsilon}_x\zeta_x^{(m)}\chi_x^{(m)}.\]
(Thus, if $m=1$, then $\zeta_x^{(1)}=\zeta_x$.) We can now state the main
result of this section.
\end{rema}

\begin{prop} \label{p1} In the setting of \ref{r21}, \ref{r22}, we have
$\zeta_x^{(m)}=\zeta_x^m$ for all $x\in X(W)$.
\end{prop}

A result of this kind is implicitly contained in Lusztig \cite[\S 14]{L2c} 
and Shoji \cite[\S 2 and 5.19]{S2}; the proof will be given in \ref{r25}. 
First, we need some preparations.

\begin{rema} \label{r23} We recall some constructions from \cite[\S 12,
\S 13]{L2c}. For any $w\in W$, we assume chosen once and for all a
representative $\dot{w}\in N_G(T)$. There is a corresponding complex 
$\bar{K}_{\dot{w}}^{\mathcal{L}_0}\in \mathcal{D}G$ as defined in 
\cite[12.1]{L2c}. Then $\hat{G}^{\text{un}}$ can also be characterized as 
the set of isomorphism classes of simple perverse sheaves on $G$ which 
occur as constituents of a perverse cohomology sheaf ${^p\!H}^i(
\bar{K}_{\dot{w}}^{\mathcal{L_0}})$ for some $w\in W$ and some $i\in\Z$. 
Furthermore, for each $i$, there is a natural isomorphism 
\[\varphi_{i,\dot{w}}\colon F^*({^p\!H}^i(\bar{K}_{\dot{w}}^{\mathcal{L}_0}))
\stackrel{\sim}{\rightarrow} {^p\!H}^i(\bar{K}_{\dot{w}}^{\mathcal{L}_0}); 
\qquad \mbox{see \cite[12.2, 13.8]{L2c}}.\] 
The advantage of using 
$\bar{K}_{\dot{w}}^{\mathcal{L}_0}$ instead of $K_w^{\mathcal{L}_0}$ (see 
\ref{r13}) is that $\bar{K}_{\dot{w}}^{\mathcal{L}_0}$ is semisimple
(see \cite[12.8]{L2c}). Let us now fix $i$, $w$ and denote $K:={^p\!H}^i
(\bar{K}_{\dot{w}}^{\mathcal{L}_0})$. We also set $\varphi:=\varphi_{i,
\dot{w}} \colon F^*K\stackrel{\sim}{\rightarrow}K$ and recall that, for each 
$x\in X(W)$, we are given $\phi_x\colon F^*A_x \stackrel{\sim}{\rightarrow}
A_x$. Now, following Lusztig \cite[13.8.2]{L2c}, there is a canonical 
isomorphism
\[K\cong \bigoplus_{x\in X(W)} (A_x\otimes V_x)\]
where $V_x$ are finite-dimensional vector spaces over $\overline{\Q}_\ell$
endowed with linear maps $\psi_x\colon V_x\rightarrow V_x$ such that, under 
the above direct sum decomposition, the map $\phi_x\otimes \psi_x$ 
corresponds to the given $\varphi\colon F^*K\stackrel{\sim}{\rightarrow} K$. 
More precisely, $V_x$ and $\psi_x$ are as follows (cf.\ \cite[10.4]{L2b} 
and \cite[3.5]{L3}). We have $V_x=\mbox{Hom}(A_x,K)$ and 
\[ \psi_x(v)=\varphi\circ F^*(v)\circ \phi_x^{-1} \qquad\mbox{for
$v\in V_x$},\]
where $F^*(v)\in \mbox{Hom}(F^*A_x,F^*K)$ is the map induced by $v\colon 
A_x\rightarrow K$. 
\end{rema}

\begin{thm}[Lusztig \protect{\cite[13.10, 14.14]{L2c}}] \label{r24} In the
setting of \ref{r23}, all eigenvalues of $\psi_x\colon V_x\rightarrow V_x$
are equal to $\zeta_x q^{(i-\dim G)/2}$.
\end{thm}

More precisely, Lusztig first shows in \cite[13.10]{L2c} that there is
a constant $0\neq \xi_x\in\overline{\Q}_\ell$ (which only depends on 
$\phi_x$, the choice of a square root of $q$ and the choice of the
representatives~$\dot{w}$) such that, for any $i$ and~$w$, all eigenvalues 
of $\psi_x\colon V_x\rightarrow V_x$ are equal to $\xi_x q^{(i-\dim G)/2}$.
(The ``cleanness'' assumption in \cite[13.10]{L2c} holds in general by 
\cite{L10}.) It is then shown in \cite[14.14]{L2c} that 
\[\rho_\epsilon=(-1)^{\dim G} \sum_{x\in X(W)} \xi_x(A_x:
K_\epsilon^{\mathcal{L}_0}) \chi_x \qquad \mbox{for all $\epsilon
\in\Irr(W)$}.\]
(There are also coefficients $\nu(A_x)$ in the formula in \cite[14.14]{L2c} 
but these are all equal to~$1$ in our situation.) Finally, a comparison with 
the formula in \ref{r18} implies that $\xi_x=\zeta_x$ for all $x\in X(W)$.

\begin{rema} \label{r25} {\it Proof of Proposition~\ref{p1}.} Let
$x\in X(W)$. We also fix $i,w$ and place ourselves in the setting
of \ref{r23}. The whole discussion there can be repeated with $F$ replaced
by $F^m$. As in \ref{r22}, we have isomorphisms
\[\tilde{\phi}_x^{(m)}\colon (F^*)^mA_x\stackrel{\sim}{\rightarrow} A_x
\qquad \mbox{and}\qquad {\phi}_x^{(m)}\colon (F^m)^*A_x
\stackrel{\sim}{\rightarrow} A_x.\]
Analogously, $\varphi\colon F^*K\stackrel{\sim}{\rightarrow} K$ induces
isomorphisms 
\[\tilde{\varphi}^{(m)}\colon (F^*)^mK\stackrel{\sim}{\rightarrow} K\qquad 
\mbox{and}\qquad \varphi^{(m)}\colon (F^m)^*K\stackrel{\sim}{\rightarrow}
K.\]
Let us consider again the canonical isomorphism 
\[K\cong \bigoplus_{x\in X(W)} (A_x\otimes V_x)\qquad\mbox{where} \qquad
V_x=\mbox{Hom}(A_x,K).\]
Here, as before, each $V_x$ is endowed with a linear map $\psi_x^{(m)}\colon 
V_x \rightarrow V_x$ such that 
\[\psi_x^{(m)}(v)=\varphi^{(m)} \circ (F^m)^*(v)\circ (\phi_x^{(m)})^{-1} 
\qquad \mbox{for all $v\in V_x$}.\]
By Theorem~\ref{r24}, the scalar $\zeta_x^{(m)}$ is determined by the 
eigenvalues of $\psi_x^{(m)}$. 

Now, a simple induction on $m$ shows that 
\[\psi_x^m(v)=\tilde{\varphi}^{(m)} \circ (F^*)^m(v)\circ 
(\tilde{\phi}_x^{(m)})^{-1} \qquad \mbox{for all $v\in V_x$}.\]
Next, we use again that we have canonical isomorphisms $(F^*)^mA_x\cong 
(F^m)^*A_x$ and $(F^*)^mK\cong (F^m)^*K$. Under these isomorphisms,
the above map $(F^*)^m(v)$ corresponds to the map $(F^m)^*(v)\colon 
(F^m)^*A_x \rightarrow (F^m)^*K$. Hence, we also have 
\[\psi_x^m(v)=\varphi^{(m)} \circ (F^m)^*(v)\circ (\phi_x^{(m)})^{-1} 
\qquad \mbox{for all $v\in V_x$}.\]
Thus, we have $\psi_x^{(m)}=\psi_x^m$ for all $x\in X(W)$ and so the 
eigenvalues of $\psi_x^{(m)}$ are obtained by raising the eigenvalues 
of $\psi_x$ to the $m$-th power. It remains to use Theorem~\ref{r24}. 
This completes the proof of Proposition~\ref{p1}. \qed
\end{rema}

\begin{rem} \label{exp1} We can, and will assume that $G$ is defined and 
split over the prime field $\F_p$ of~$k$. Let $F_0\colon G\rightarrow G$ 
be the corresponding Frobenius map. If $q=p^f$ where $f\geq 1$, then $F=
F_0^f$. Hence, Proposition~\ref{p1} means that it will be sufficient to 
determine the scalars $\zeta_x$ $(x\in X(W)$) for the group $G^{F_0}=
G(\F_p)$. For specific values of~$p$ (e.g., bad primes $p=2,3,5$), we may 
then use ad hoc information which is available, for example, in the Cambridge
{\sf ATLAS} \cite{atl}, or via computer algebra methods (using {\sf GAP}
\cite{gap}, {\sf CHEVIE} \cite{jmich}). This is the basis for the 
discussion of the examples below.
\end{rem}

\section{Cuspidal character sheaves and small rank examples} \label{sec3}

We keep the notation of the previous section; in particular, we label
all objects by the parameter set $X(W)$ as in \ref{r18}. By \cite[3.5]{L3}, 
the computation of the scalar $\zeta_x$ can be reduced to the case where 
$A_x$ is a cuspidal character sheaf (in the sense of \cite[Def.~3.10]{L2a}). 
So let us look in more detail at this case.

\begin{rema} \label{exp2} Assume that $Z(G)=\{1\}$. Let $x\in X(W)$ be such
that $A_x\in\hat{G}^{\text{un}}$ is cuspidal. Then there exists an $F$-stable 
conjugacy class $C$ of $G$ and an irreducible, $G$-equivariant 
$\overline{\Q}_\ell$-local system $\cE$ on $C$ such that $F^*\cE\cong 
\cE$ and $A=\mbox{IC}(\overline{C},\cE)[\dim C]$; see \cite[3.12]{L2a}. In 
particular, $\mbox{supp}(A_x)=\overline{C}$ and so $\hat{\varepsilon}_{A_x}=
(-1)^{\dim G-\dim C}$. Let us fix $g_1\in C^F$ and set $A(g_1):=C_G(g_1)/
C_G^\circ(g_1)$. Then $F$ induces an automorphism $\gamma\colon A(g_1)
\rightarrow A(g_1)$. We further assume that:
\begin{itemize}
\item[($*$)] the local system $\cE$ is one-dimensional and, hence, 
corresponds to a $\gamma$-invariant linear character $\lambda\colon 
A(g_1) \rightarrow \overline{\Q}_\ell^\times$ (via \cite[19.7]{Ldisc4}).  
\end{itemize}
(This assumption will be satisfied in all examples that we consider.)
We form the semidirect product $\tilde{A}(g_1)=A(g_1)\rtimes \langle 
\gamma\rangle$ such that, inside $\tilde{A}(g_1)$, we have the 
identity $\gamma(a)=\gamma a\gamma^{-1}$ for all $a\in A(g_1)$. By ($*$), 
we can canonically extend $\lambda$ to a linear character 
\[\tilde{\lambda}\colon \tilde{A}(g_1) \rightarrow 
\overline{\Q}_\ell^\times,\qquad  a\gamma\mapsto \lambda(a).\]
For each $a\in A(g_1)$ we have a corresponding element $g_a\in C^F$, 
well-defined up to conjugation within $G^F$. (We have $g_a=hg_1h^{-1}$
where $h\in G$ is such that $h^{-1}F(h)\in C_G(g_1)$ has image $a\in 
A(g_1)$.) We define a class function $\chi_{g_1,\lambda} \colon G^F
\rightarrow \overline{\Q}_\ell$ by
\[\chi_{g_1,\lambda}(g)=\left\{\begin{array}{cl} q^{(\dim G-\dim C)/2}
\lambda(a) & \quad\mbox{if $g=g_a$ for some $a\in A(g_1)$},\\ 0 & \quad
\mbox{if $g\not\in C^F$}.\end{array}\right.\]
Now, we can choose an isomorphism $F^*\cE \stackrel{\sim}{\rightarrow}\cE$ 
such that the induced map on the stalk $\cE_{g_1}$ is scalar multiplication 
by $q^{(\dim G-\dim C)/2}$. Then this isomorphism canonically induces an 
isomorphism $\phi_x\colon F^*A_x \stackrel{\sim}{\rightarrow} A_x$ which
satisfies the requirements in \ref{r15}, and we have $\chi_x=\chi_{A_x,
\phi_x}=\chi_{g_1,\lambda}$. (This follows from the fact that $A_x$ is 
``clean'' \cite{L10}, using the construction in \cite[19.7]{Ldisc4}.) With
this choice of $\phi_x$, we also have for all $m\geq 1$: 
\[ \chi_x^{(m)}(g_a^{(m)})=\left\{\begin{array}{cl} q^{m(\dim G-\dim C)/2}
\lambda(a) & \quad\mbox{if $g=g_a^{(m)}$ for some $a\in A(g_1)$},\\ 0 & 
\quad \mbox{if $g\not\in C^{F^m}$}.\end{array}\right.\]
(Here, $g_a^{(m)}=hg_1h^{-1}$ where now $h\in G$ is such that $h^{-1}F^m(h)
\in C_G(g_1)$ has image $a\in A(g_1)$; see again \cite[19.7]{Ldisc4}.) The 
identity in Definition~\ref{s1} now reads:
\[ R_x=\sum_{y\in X(W)} \delta_y\{y,x\}\rho_y=(-1)^{\dim C}
\zeta_x\chi_{g_1,\lambda}.\]
\end{rema}

\begin{rem} \label{rem1} Let $G$ be of (split) classical type. Then Shoji 
has shown that we always have $\zeta_x=1$ for cuspidal $A_x\in
\hat{G}^{\text{un}}$; see \cite[Prop.~6.7]{Sclass} for $p\neq 2$, and 
\cite[Theorem~6.2]{S7} for $p=2$. Note that this involves, in each case, 
the choice of a particular representative in the conjugacy class supporting 
$A_x$. Since classical groups of low rank appear as Levi subgroups in 
groups of exceptional type, it will be useful to work out explicitly the 
relevant identities $R_x=(-1)^{\dim C} \zeta_x\chi_{(g_1,\lambda)}$ for 
$G$ of type $C_2$, $D_4$ and $p=2$. This also provides a good illustration 
for: (a) the role of the choice of a class representative as above and (b) 
the strategy that we will employ when dealing with groups of exceptional
type. 
\end{rem}

\begin{exmp} \label{sp4} Let $G=\mbox{Sp}_4(k)$ be the $4$-dimensional 
symplectic group. Then $G=\langle x_\alpha(t)\mid \alpha\in \Phi, t\in k
\rangle$ with root system $\Phi=\{\pm a,\pm b,\pm (a+b),\pm (2a+b)\}$. The 
Weyl group $W=\langle s_a,s_b\rangle$ is dihedral of order~$8$ and we have
\[ \Irr(W)=\{1_W,\text{sgn},\text{sgn}_a,\text{sgn}_b,r\}\]
where $1_W$ is the trivial character, $\text{sgn}$ is the sign character,
$r$ has degree~$2$, and $\text{sgn}_a,\text{sgn}_b$ are linear characters 
such that $\text{sgn}_a(s_a)=1$, $\text{sgn}_a(s_b)=-1$, $\text{sgn}_b
(s_a)=-1$, $\text{sgn}_b(s_b)=1$. By \cite[p.~468]{C2}, we have 
\[X(W)=\{x_1, x_{\text{sgn}}, x_{\text{sgn}_a}, x_{\text{sgn}_b}, x_{r},
x_0\}\]
where $\rho_{1_W}(1)=1$, $\rho_{\text{sgn}}(1)=q^4$, $\rho_{\text{sgn}_a}
(1)=\rho_{\text{sgn}_b}(1)=\frac{1}{2}q(q^2+1)$, $\rho_{r}(1)=\frac{1}{2}
q(q+1)^2$ and $\rho_{x_0}(1)=\frac{1}{2}q(q-1)^2$. By the explicit 
description of the Fourier matrices in \cite[p.~471]{C2}, we find that 
\begin{center}
$R_{x_0}=\frac{1}{2}(\rho_{r}-\rho_{\text{sgn}_a}-\rho_{\text{sgn}_b}+
\rho_{x_0})$.
\end{center}
If $q$ is odd, then the identification of $R_{x_0}$ with a characteristic 
function of a cuspidal character sheaf is explained in the appendix of 
Srinivasan \cite{Sr94}. Now assume that $q=2^f$ where $f\geq 1$. Then, by 
\cite[22.2]{L2d}, there is a unique cuspidal character sheaf $A_0$ on $G$, 
and it is contained in $\hat{G}^{\text{un}}$. By the explicit description 
in \cite[2.7]{LuSp1}, we have $A_0=A_{x_0}$ and $A_0=\mbox{IC}(\overline{C},
\cE)[\dim C]$ where $C$ is the class of regular unipotent elements and 
$\cE\not\cong \overline{\Q}_\ell$; we have $\dim G=10$ and $\dim C=8$. Let 
us fix 
\[g_1=x_a(1)x_b(1) \in C^F.\]
One checks that $g_1$ has order $4$ and that $g_1$ is conjugate in $G^F$ 
to $g_1^{-1}=x_b(1)x_a(1)$. Furthermore, $A(g_1)\cong \Z/2\Z$ is abelian 
and $F$ acts trivially on $A(g_1)$. Let $\lambda$ be the non-trivial 
character of $A(g_1)$. Then, as in Example~\ref{exp2}, we obtain: 
\[\chi_{g_1,\lambda}(g)=\left\{\begin{array}{cl} q &\quad\mbox{if $g=g_1$},\\
-q & \quad \mbox{if $g=g_1'$},\\ 0 & \quad \mbox{if $g\not\in C^F$},
\end{array}\right.\]
where $g_1'\in C^F$ corresponds to the non-trivial element of $A(g_1)$.
We now have $R_{x_0}=\zeta_{x_0}\chi_{g_1,\lambda}$. In order
to determine $\zeta_{x_0}$ it is sufficient, by Remark~\ref{exp1}, to 
consider the case where~$q=2$. But $\mbox{Sp}_4(\F_2)$ is isomorphic 
to the symmetric group $\mathfrak{S}_6$; an explicit isomorphism is 
described in \cite[9.21]{HupI}. One checks that, under this isomorphism, 
$g_1=x_a(1) x_b(1)$ corresponds to an element of cycle type $(4,2)$ in 
$\mathfrak{S}_6$. We also need to identify $\rho_{r}$, $\rho_{\text{sgn}_a}$, 
$\rho_{\text{sgn}_b}$, $\rho_{x_0}$ in the character table of 
$\mathfrak{S}_6$. Now, $B(\F_2)$ is a Sylow $2$-subgroup of $\mbox{Sp}_4
(\F_2)$. Working out the character of the permutation representation on the 
cosets of this subgroup, one can identify the $5$ characters which are of 
the form $\rho_\epsilon$ for some $\epsilon\in \Irr(W)$. Looking also at 
character degrees, we can then identify $\rho_{\text{sgn}}$ and the sum 
$\rho_{\text{sgn}_a}+\rho_{\text{sgn}_b}$; finally, $\rho_{x_0}$ corresponds
to the sign character of $\mathfrak{S}_6$. By inspection of the table of 
$\mathfrak{S}_6$, we find that $R_{x_0}(g_1)=2$ (for $q=2$) and, hence, 
$\zeta_{x_0}=1$ for any $q=2^f$ (using Proposition~~\ref{p1}).~---~Of 
course, this could also be deduced from the explicit knowledge of the 
``generic'' character table of $G^F=\mbox{Sp}_4(\F_q)$ for any $q=2^f$ 
(see Enomoto \cite{En}). But the point is that, once $\zeta_{x_0}$ is known 
in advance, the task of computing such a generic character table is 
considerably simplified! 
\end{exmp}

\begin{exmp} \label{so8} Let $G=\mbox{SO}_8(k)$ be the $8$-dimensional 
special orthogonal group. The Weyl group $W$ has $13$ irreducible 
characters, which are labelled by certain pairs of partitions. By 
\cite[p.~471]{C2}, we have 
\[ X(W)=\{x_\epsilon\mid \epsilon\in\Irr(W)\} \cup\{x_0\}.\]
By the explicit description of the Fourier matrices in \cite[p.~472]{C2}, 
we find that 
\begin{center}
$R_{x_0}=\frac{1}{2}(\rho_{(21,1)}-\rho_{(22,\varnothing)}-
\rho_{(2,11)}+ \rho_{x_0})$;
\end{center}
here, $(1,21)$, $(\varnothing,22)$, $(11,2)$ indicate irreducible 
characters $\epsilon\in\Irr(W)$ (as in \cite[p.~449]{C2}). If $q$ is odd, 
then the values of $R_{x_0}$ are explicitly computed in 
\cite[Prop.~4.5]{GP1} (and this provides an identification of $R_{x_0}$ 
with a characteristic function of a cuspidal character sheaf). Now assume 
that $q=2^f$ where $f\geq 1$. Then, by \cite[22.3]{L2d}, there is a unique
cuspidal character sheaf $A_0$ on $G$, and it is contained in 
$\hat{G}^{\text{un}}$. Again, by the explicit description in 
\cite[3.3]{LuSp1}, we have $A_0=A_{x_0}=\mbox{IC}(\overline{C},\cE)
[\dim C]$ where $C$ is the class of regular unipotent elements and 
$\cE\not\cong \overline{\Q}_\ell$; we have $\dim G=28$ and $\dim C=24$. 
Let us fix 
\[g_1=x_a(1)x_b(1)x_c(1)x_d(1) \in C^F\]
where $\{a,b,c,d\}$ is a set of simple roots in the root system of type
$D_4$. One checks that $g_1$ has order $8$ and that, if $a',b',c',d'$ is any 
permutation of $a,b,c,d$, then $g_1$ is conjugate in $G^F$ to $x_{a'}(1)
x_{b'}(1)x_{c'}(1)x_{d'}(1)$. Furthermore, $A(g_1)\cong \Z/2\Z$ is abelian 
and $F$ acts trivially on $A(g_1)$. Let $\lambda$ be the non-trivial 
character of $A(g_1)$. As above, we obtain: 
\[\chi_{g_1,\lambda}(g)=\left\{\begin{array}{cl} q^2 &\quad\mbox{if 
$g=g_1$},\\ -q^2 & \quad \mbox{if $g=g_1'$},\\ 0 & \quad \mbox{if 
$g\not\in C^F$},
\end{array}\right.\]
where $g_1'\in C^F$ corresponds to the non-trivial element of $A(g_1)$.
Now we have $R_{x_0}=\zeta_{x_0}\chi_{g_1,\lambda}$. In order to show 
that $\zeta_{x_0}=1$, we can use the known character table of 
$\mbox{SO}_8^+(\F_2)$; see the {\sf ATLAS} \cite[p.~85]{atl}. In fact, using 
an explicit realization in terms of orthogonal $8\times 8$-matrices, one can
create $\mbox{SO}_8^+(\F_2)$ as a matrix group in {\sf GAP} and simply 
re-calculate that table using the {\tt CharacterTable} function. The 
advantage of this re-calculation is that {\sf GAP} also computes a list of 
representatives of the conjugacy classes of $\mbox{SO}_8^+(\F_2)$. So one 
can identify the class to which~$g_1$ belongs. Arguing as in the previous
example, one can identify the characters $\rho_{x_0}$, $\rho_{(21,1)}$ and 
the sum $\rho_{(22,\varnothing)}+\rho_{(2,11)}$ in the table of 
$\mbox{SO}_8^+(\F_2)$. (We omit the details.) In this way, one finds that 
$R_{x_0} (g_1)=4$ (for $q=2$), as required. 
\end{exmp}

\begin{rem} \label{rsplit} Assume that $Z(G)=\{1\}$, as above. Let $x\in 
X(W)$ be such that $A_x$ is cuspidal and let $C$ be the $F$-stable
conjugacy class of $G$ such that $\mbox{supp}(A_x)=\overline{C}$. The above 
examples highlight the importance of singling out a specific representative
$g_1\in C^F$ in order to determine a characteristic function $\chi_x$ of 
$A_x$ and the scalar~$\zeta_x$. This problem is, of course, not a new one. 
If $C$ is a unipotent class and $p$ is a good prime for $G$, there is a 
notion of ``split'' elements in $C^F$ which solves this problem in almost
all cases; see Shoji's survey \cite[\S 5]{S1}. Despite of much further work
(e.g., Shoji \cite{S6}), the question of finding general conditions which 
single out a distinguished representative $g_1 \in C^F$ appears to be open. 
In the above examples (and those below), we are able to choose a 
representative~$g_1\in C^F$ according to the following principles:
\begin{itemize}
\item $g_1$ belongs to $C^{F_0}=C(\F_p)$ (cf.\ Remark~\ref{exp1}) and is 
conjugate in $G(\F_p)$ to all powers $g_1^n$ where $n\in \Z$ is coprime to 
the order of~$g_1$.
\item The unipotent part of $g_1$ has a ``short'' expression in 
terms of the Chevalley generators $x_\alpha(t)$ of $G$ where $\alpha$ is 
a root and $t\in\F_p$. 
\end{itemize}
These principles also work in the further examples discussed in \cite{GeHe}.
\end{rem}
\section{Type $F_4$ in characteristic~$2$} \label{secF4}

Throughout this section, we assume that $G$ is simple of type $F_4$ and 
$p=2$. We have $G=\langle x_\alpha(t)\mid \alpha\in\Phi,t\in k\rangle$ where 
$\Phi$ is the root system of $G$ with respect to~$T$. Let $\{\alpha_1,\alpha_2,
\alpha_3,\alpha_4\}$ be the set of simple roots with respect to~$B$, where 
we choose the notation such that $\alpha_1,\alpha_2$ are long roots, 
$\alpha_3, \alpha_4$ are short roots, and $\alpha_2,\alpha_3$ are not 
orthogonal. (This coincides with the conventions in Shinoda \cite{Shi}, 
where the conjugacy classes of $G$ are determined.)
Also note that, since $p=2$, we do not have to worry about the precise 
choice of a Chevalley basis in the underlying Lie algebra. Our aim is
to determine the exact relation between characteristic functions of cuspidal 
unipotent character sheaves and almost characters in this case, using 
the approach illustrated in the previous section. For this purpose, we 
essentially rely on the knowledge of the character table of $F_4(\F_2)$; 
see the {\sf ATLAS} \cite[p.~167]{atl}. (The correctness of this table has 
been verified independently in \cite{veri}, and it is available in the 
library of {\sf GAP} \cite{gap}.) We shall also rely on a number of explicit 
computational results obtained through general algorithms and functions 
concerning matrix groups and character tables in {\sf GAP}.

\begin{rema} \label{r51} There are seven cuspidal characters sheaves 
$A_1,\ldots,A_7$ on $G$, described in detail by Shoji \cite[\S 7]{S2}
(based on earlier work of Lusztig and Spaltenstein). Firstly, they are 
all contained in $\hat{G}^{\text{un}}$; see \cite[7.6]{S2}. Secondly, as 
in \ref{exp2}, each $A_j$ corresponds to a pair $(g_1,\lambda)$ where 
$g_1\in G^F$ and $\lambda\in \Irr(A(g_1))$ is a non-trivial linear 
character of $A(g_1)$; see \cite[\S 7.2]{S2}. Here, in all seven cases,
$F$ acts trivially on $A(g_1)$. The correspondences $A_j \leftrightarrow 
(g_1,\lambda)$ are given as follows.
\begin{itemize}
\item[(a)] $A_j\leftrightarrow (u,\lambda_j)$ ($j=1,2$) where $u\in G^F$ 
is regular unipotent, $\dim C_G(u)=4$; furthermore, $A(u)\cong \Z/4\Z$ is 
generated by the image $\bar{u}\in A(u)$, and $\lambda_j$ are the linear 
characters such that $\lambda_1(\bar{u})=i$, $\lambda_2(\bar{u})=-i$, 
where $i=\sqrt{-1}$ is a fixed fourth root of unity.
\item[(b)] $A_3\leftrightarrow (u,\lambda)$ where $u\in G^F$ is
unipotent such that $\dim C_G(u)=6$ and $A(u)\cong \Z/2\Z$; furthermore,
$\lambda$ is the non-trivial character of $A(u)$.
\item[(c)] $A_4\leftrightarrow (u,\lambda)$ where $u\in G^F$ is
unipotent such that $\dim C_G(u)=8$ and $A(u)$ is isomorphic to the 
dihedral group $D_8$; furthermore, $\lambda$ corresponds to the sign 
character of $D_8$.
\item[(d)] $A_5\leftrightarrow (u,\lambda)$ where $u\in G^F$ is
unipotent such that $\dim C_G(u)=12$ and $A(u)$ is isomorphic to the 
symmetric group $\mathfrak{S}_3$; furthermore, $\lambda$ corresponds 
to the sign character of $\mathfrak{S}_3$.
\item[(e)] $A_j\leftrightarrow (su,\lambda_j)$ ($j=6,7)$ where $s\in G^F$ 
is semisimple, with $C_G(s)$ isogenous to $\mbox{SL}_3(k)\times \mbox{SL}_3
(k)$, and $u \in C_G(s)^F$ is regular unipotent; we have $\dim C_G(su)=4$. 
Furthermore, $A(su)\cong \Z/3\Z$ is generated by the image $\overline{su}
\in A(su)$, and $\lambda_j$ are linear characters such that $\lambda_6
(\overline{su})=\theta$, $\lambda_7(\overline{su})=\theta^2$, where 
$\theta$ is a fixed third root of unity.
\end{itemize}
(In each case (a)--(d), the conditions on $\dim C_G(u)$ and $A(u)$ uniquely
determine $g_1=u$ up to $G$-conjugacy; see Shinoda \cite[\S 2]{Shi}.
Furthermore, the class of $s$ in (e) is also uniquely determined such that
$C_G(s)$ has type $A_2\times A_2$; see \cite[\S 3]{Shi}.)
\end{rema}

\begin{table}[htbp] \caption{Supports of cuspidal character sheaves in
type $F_4$, $p=2$} \label{supp2}
\begin{center}
$\begin{array}{clcclc} \hline \text{In \ref{r51}} & 
\text{representative $g_1$}&A(g_1)&|C_G(g_1)^F| & \text{\sf GAP}\\ \hline 
\text{(a)}&u_{31}=x_{1000}(1)x_{0100}(1)x_{0010}(1)x_{0001}(1) & \Z/4\Z & 4q^4
& {\tt 16a} (76)\\ 
\text{(b)}&u_{29}=x_{0122}(1)x_{1000}(1)x_{0100}(1)x_{0010}(1) & \Z/2\Z & 
2q^6 & {\tt 8j} (47)\\
\text{(c)}&u_{24}=x_{1100}(1)x_{0120}(1)x_{0001}(1)x_{0011}(1) & D_8 & 8q^8 
& {\tt 8a} (38)\\
\text{(d)}&u_{17}=x_{1110}(1)x_{1220}(1)x_{0011}(1)x_{0122}(1) & 
\mathfrak{S}_3 &  6q^{12} & {\tt 4l} (20)\\ \text{(e)} & su_{17}=u_{17}s 
\;\text{ with $s\in G^F$ of order $3$} & \Z/3\Z & 3q^4 & {\tt 12o} (68)\\
\hline \multicolumn{5}{c}{\text{(Notation for $u_i$ as in Shinoda
\cite[\S 2]{Shi})}}\end{array}$
\end{center}
\end{table}

\begin{rema} \label{r52} Let $A_j$ be one of the cuspidal character sheaves
in \ref{r51}. In Table~\ref{supp2}, we fix a specific element $g_1\in G^F$ 
for the corresponding pair $(g_1,\lambda)$. The characteristic function 
$\chi_j:=\chi_{A_j,\phi_{A_j}}$ in \ref{exp2} will then be defined with 
respect to this choice of~$g_1$.

In the table, we use the following notation. If $\alpha$ is a positive root, 
written as $\alpha=\sum_{1\leq i\leq 4} n_i\alpha_i$, then we denote the 
corresponding root element $x_\alpha(t)$ by $x_{n_1n_2n_3n_4}(t)$. Some 
further comments about the entries of Table~\ref{supp2}.

\smallskip
(1) In the cases (a)--(d), $g_1$ is unipotent; let $C$ be the conjugacy class
of $g_1$ in $G$. Then we choose $g_1$ to be that representative of $C^F$ in 
Shinoda's list \cite[\S 2]{Shi}, which has an expression as a product of root
elements $x_\alpha(t)$ whose description does not involve the elements $\eta,
\zeta\in k$ defined in \cite[2.2]{Shi}. All the coefficients in the 
expression for $g_1$ are equal to~$1$, and so $g_1\in F_4(\F_2)$.

\smallskip
(2) Let $g_1\in F_4(\F_2)$ be as in (1). Then we will need to know to which
class of the {\sf GAP} character table of $F_4(\F_2)$ this element belongs.
This problem is easy for some cases, e.g., $g_1=u_{17}$ since there is a 
unique class in the {\sf GAP} table with centraliser order $6\cdot 2^{12}=
24576$. On the other hand, if $g_1=u_{31}$, then the classes with 
centraliser order $4\cdot 2^4=64$ are quite difficult to distinguish
in the {\sf GAP} table, especially the two classes {\tt 16a}, {\tt 16b}. 
In such cases, we use the following (computational) argument. Using the 
explicit $52$-dimensional matrix realization of $F_4(\F_2)$ from Lusztig 
\cite[2.3]{L12} (see also \cite[4.10]{mylie}), we create $F_4(\F_2)$ as a 
matrix group in {\sf GAP}. Then we consider the subgroup
\[P:=\langle x_{\pm 1000}(1), x_{\pm 0100}(1),x_{\pm 0010}(1), x_{0001}(1)
\rangle\subseteq F_4(\F_2).\]
The {\sf GAP} function 
{\tt CharacterTable} computes the character table of $P$ (by a general
algorithm, without using any specific properties of $P$), together with 
a list of representatives of the conjugacy classes. There are $214$ 
conjugacy classes of $P$ and, again by standard algorithms, {\sf GAP} 
can find out to which of these $214$ classes any given element of $P$ 
belongs. Now the function {\tt PossibleClassFusions} determines all possible 
fusions of the conjugacy classes of $P$ into the {\sf GAP} character table 
of $F_4(\F_2)$. (Here, a ``possible class fusion'' is a map which assigns 
to each conjugacy class of $P$ one of the $95$ class labels {\tt 1a}, 
{\tt 2a}, $\ldots$, {\tt 30b} of the {\sf GAP} table of $F_4(\F_2)$, such 
that certain conditions are satisfied which should hold if the map is a 
true matching of the classes of $P$ with those of the {\sf GAP} table;
see the help menu of {\tt PossibleClassFusions} for further details.) 
As might be expected (since there are non-trivial table automorphisms
of the character table of $F_4(\F_2)$), the class fusion is not unique; in 
fact, it turns out that there are $16$ possible fusion maps. But, if $g_1 
\in \{u_{17}, u_{24}, u_{29}, u_{31}\}$, then $g_1$ is mapped to the same 
class label of $F_4(\F_2)$, under each of the $16$ possibilities. Thus, the 
fusion of $g_1$ is uniquely determined, and this is the entry in the last 
column of Table~\ref{supp2}.

\smallskip
(3) Let again $g_1\in F_4(\F_2)$ be as in (1). Having identified the
class of $g_1$ in the {\sf GAP} table, we can simply check by inspection 
that $\chi(g_1) \in\Z$ for all $\chi\in\Irr(F_4(\F_2))$. Hence, $g_1$ is 
conjugate within $F_4(\F_2)$ to each power $g_1^n$ where $n$ is a positive 
integer coprime to the order of~$g_1$; in particular, $g_1$, $g_1^{-1}$ 
are conjugate in $F_4(\F_2)$. 

\smallskip
(4) Now let $g_1=su$ be as in case~(e). Let $C$ be the $G$-conjugacy class 
of~$su$. Since $s$ has order~$3$ and $u$ has order~$4$, the 
element $g_1=su$ has order~$12$; furthermore, $|C_G(g_1)^F|=3q^4$. All this 
also works for the base case where $q=2$ and so we can assume that $g_1 \in 
C(\F_2)\subseteq F_4(\F_2)$. Then $C(\F_2)$ splits into~$3$ classes in 
$F_4(\F_2)$; all elements in these three classes have order~$12$ and
centraliser order~$3\cdot 2^4=48$. By inspection of the {\sf GAP} table, we 
see that these three classes must be {\tt 12o}, {\tt 12p}, {\tt 12q}; here,
{\tt 12o} has the property that all character values on this class are 
integers. We now choose $g_1$ to be in {\tt 12o}; then $g_1$ is conjugate 
in $F_4(\F_2)$ to all powers $g_1^n$ where $n$ is a positive integer coprime 
to the order of~$g_1$. In particular, $g_1$ is conjugate in $G^F$ (even in 
$F_4(\F_2)$) to $g_1^5=s^{-1}u$ and, consequently, $s,s^{-1}$ are conjugate 
in $C_G(u)^F$. Now the {\sf GAP} table also shows that $g_1^3=u^3=u^{-1}$ 
belongs to {\tt 4l}. We have seen in (2), (3) that both $u_{17}$ and 
$u_{17}^{-1}$ belong to {\tt 4l}. Hence, we can assume that $u=u_{17}$. It 
then also follows from the above results (and Sylow's Theorem) that, if 
$t\in C_G(u)^F$ is any element of order~$3$, then~$tu$ is conjugate 
to~$su$ in $G^F$. 
\end{rema}

\begin{table}[htbp] \caption{Unipotent characters in the family $\cF_0$ for
type $F_4$, $q=2$} \label{uni2}
\begin{center}
$\begin{array}{crc} \hline \rho_x & \rho_x(1) & 
\text{{\sf GAP}} \\\hline 
    F_4^{II}[1]&       1326 &  X.5\\ 
    F_4^{I}[1]&       21658 &  X.7  \\ 
    F_4[-1]&          63700 &  X.13  \\ 
    \phi_{1,12}''&    99450 & X.14   \\ 
    \phi_{1,12}'&     99450 & X.15   \\ 
    \{F_4[i], F_4[-i]\}&      142884 & \{X.16,X.17\}   \\ 
    \{F_4[\theta],F_4[\theta^2]\}&   183600 & \{X.20,X.21\}   \\ 
    B_2:(\varnothing.2)&          216580 &  X.22  \\ 
    B_2:(11.\varnothing)&         216580 &  X.23  \\ 
\hline \multicolumn{3}{l}{\text{(Notation for $\rho_x$ as in Carter
\cite[p.~479]{C2})}} \end{array} \qquad\qquad \begin{array}{crc} \hline
\rho_x & \rho_x(1) & \text{{\sf GAP}}\\\hline
    \phi_{6,6}''&     249900 &  X.24  \\ 
    B_2:(1.1)&          269892 &  X.25  \\ 
    \phi_{4,8}&       322218 &  X.27 \\ 
    \phi_{4,7}''&     358020 &  X.30  \\ 
    \phi_{4,7}'&      358020 &  X.31 \\ 
    \phi_{6,6}'&      519792 &  X.32  \\ 
    \phi_{9,6}''&     541450 &  X.33 \\ 
    \phi_{9,6}'&      541450 &  X.34 \\ 
    \phi_{12,4}&      584766 &  X.37 \\ 
    \phi_{16,5}&      947700 &  X.44 \\ 
\hline \end{array}$
\end{center}
\end{table}

\begin{rema} \label{r53} Let us write $\fU(G^F)=\{\rho_x\mid x\in 
X(W)\}$ and $\hat{G}^{\text{un}}=\{A_x \mid x \in X(W)\}$ as in \ref{r18}. 
Here, the parameter set $X(W)$ has $37$ elements. It is partitioned into 
$11$ ``families'', which correspond to the special characters of $W$, as 
defined in \cite[4.1, 4.2]{L1}. The corresponding Fourier matrix has a 
block diagonal shape, with one diagonal block for each family. The $7$ 
elements of $X(W)$ which label cuspidal character sheaves are all 
contained in one family, which we denote by $\cF_0$ and which consists 
of $21$ elements in total. The elements of $\cF_0$ are given by the $21$ 
pairs $(u,\sigma)$ where $u\in\fS_4$ (up to $\fS_4$-conjugacy) and 
$\sigma\in\Irr(C_{\fS_4}(u))$. The diagonal block of the Fourier matrix
corresponding to $\cF_0$ is then described in terms of $\fS_4$ and the 
pairs $(u,\sigma)$; see \cite[\S 13.6]{C2} for further details. The 
identification of the $21$ unipotent characters $\{\rho_x\mid x\in\cF_0\}$ 
with characters in the {\sf GAP} table of $F_4(\F_2)$ is given in 
Table~\ref{uni2}. (This information is available through the {\sf GAP} 
function {\tt DeligneLusztigNames}; note that this depends on fixing a
labelling of the long and short roots.) The results of Shoji 
\cite[Theorem 7.5]{S2} show that the cuspidal character sheaves $A_j$ in 
\ref{r51} are labelled as follows by pairs in $\cF_0$:
\begin{gather*}
A_1=A_{(g_4,i)}, \qquad A_2=A_{(g_4,-i)}, \qquad
A_3=A_{(g_2,\epsilon)},\\ \{A_4,A_5\}=\{A_{(1,\lambda^3)},A_{(g_2',
\varepsilon)}\},\qquad A_6=A_{(g_3,\theta)},\qquad A_7=A_{(g_3,\theta^2)}.
\end{gather*}
Here, we follow the notation in \cite[p.~455]{C2} for pairs in $\cF_0$. (Note
that, a priori, the statement of Theorem~\ref{r16} alone does not tell us 
anything about the supporting set of a cuspidal character sheaf $A_x$!)
Let us consider in detail $A_6,A_7$. (The remaining cases have been
dealt with already by Marcelo--Shinoda \cite{MaSh}; see Remark~\ref{mash}
below.) Using the $21\times 21$-times Fourier matrix printed in 
\cite[p.~456]{C2} and the labelling of the unipotent characters $\{\rho_x
\mid x \in \cF_0\}$ in terms of pairs $(u,\sigma)$ as above (see 
\cite[p.~479]{C2}), we obtain the following formulae for the unipotent 
almost characters labelled by $(g_3,\theta)$ and $(g_3,\theta^2)$:
\begin{align*}
R_{(g_3,\theta)} & =\textstyle\frac{1}{3}\bigl(\rho_{(12,4)}+F_4^{II}[1]-
\rho_{(6,6)'}-\rho_{(6,6)''}+ 2F_4[\theta]-F_4[\theta^2]\bigr),\\
R_{(g_3,\theta^2)} & =\textstyle\frac{1}{3}\bigl(\rho_{(12,4)}+F_4^{II}[1]-
\rho_{(6,6)'}-\rho_{(6,6)''} -F_4[\theta]+2F_4[\theta^2]\bigr).
\end{align*}
By \ref{r52} and Theorem~\ref{r16}, we have $R_{(g_3,\theta)}=
\zeta_{(g_3,\theta)}\chi_6$ and $R_{(g_3,\theta^2)}=\zeta_{(g_3,\theta^2)}
\chi_7$. 
\end{rema}

\begin{prop} \label{p2} With notation as in \ref{r51}--\ref{r53}, we have
$\zeta_{(g_3,\theta)}=\zeta_{(g_3,\theta^2)}=1$; hence, $R_{(g_3,\theta)}=
\chi_6$ and $R_{(g_3,\theta^2)}=\chi_7$.
\end{prop} 

\begin{proof} Let $C$ be the conjugacy class of $g_1=su_{17}$ in $G$ (cf.\ 
Table~\ref{supp2}). Recall from \ref{r52} that the values of $\chi_6$ and 
$\chi_7$ on $C^F$ are given as follows:
\[\chi_6(g)= \left\{\begin{array}{cl} q^4 & \mbox{ (if $g=g_1$)}\\
 q^4 \theta& \mbox{ (if $g=g_{\bar{g}_1}$)}\\ q^4\theta^2 & \mbox{ (if $g=
g_{\bar{g}_1^2}$)}\end{array}\right.
\quad \mbox{and}\qquad \chi_7(g)= \left\{\begin{array}{cl} q^4 & \mbox{ (if 
$g=g_1$)}\\ q^4 \theta^2& \mbox{ (if $g=g_{\bar{g}_1}$)}\\ q^4\theta & 
\mbox{ (if $g=g_{\bar{g}_1^2}$)} \end{array}\right.\]
In order to prove that $\zeta_{(g_3,\theta)}=\zeta_{(g_3,\theta^2)}=1$ it
is sufficient, by Proposition~\ref{p1}, to consider the base case where 
$q=2$. So, in the {\sf GAP} table of $F_4(\F_2)$, we form the above linear 
combinations of the irreducible characters, giving $R_{(g_3,\theta)}$ and 
$R_{(g_3,\theta^2)}$ for the group $F_4(\F_2)$. We find that these linear 
combinations have value $16$ on the class {\tt 12o}, which contains 
$g_1=su_{17}\in C^F$. (Note that this does not depend on how we match 
$F_4[\theta]$, $F_4[\theta^2]$ with {\tt X.20}, {\tt X.21} in the 
{\sf GAP} table, since {\tt X.20} and {\tt X.21} are algebraically 
conjugate and all values on {\tt 12o} are integers.) Comparing with 
the above formulae we see that $\zeta_{(g_3,\theta)}= 
\zeta_{(g^3,\theta^2)}=1$, as desired.
\end{proof}

\begin{rem} \label{mash} The remaining $A_1,\ldots,A_5$ have already been 
dealt with by Marcelo--Shinoda \cite[Theorem~4.1]{MaSh}, by completely 
different (and, in our opinion, computationally somewhat more complicated) 
methods. In particular, they showed that 
\[ A_4=A_{(g_3',\varepsilon)}, \quad A_5=A_{(1,\lambda^3)} \quad 
\mbox{and} \quad \zeta_{A_1}= \zeta_{A_2}= \zeta_{A_3}= \zeta_{A_4}=
\zeta_{A_5}=1.\]
In our setting, this can also be shown by exactly the same kind of argument 
as in the proof of Proposition~\ref{p2}. 
\end{rem}

\section{Type $E_6$ in characteristic~$2$} \label{secE6}

Throughout this section, we assume that $G$ is simple of adjoint type 
$E_6$ and $p=2$. We have $G=\langle x_\alpha(t)\mid \alpha\in\Phi,t\in 
k\rangle$ where $\Phi$ is the root system of $G$ with respect to~$T$. Let 
$\{\alpha_1,\alpha_2, \alpha_3,\alpha_4,\alpha_5,\alpha_6\}$ be the set 
of simple roots with respect to~$B$, where the labelling is chosen as
in the following diagram:
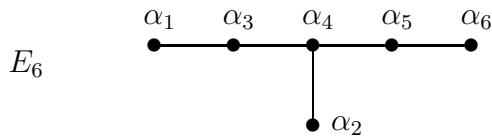
\begin{center}
\begin{picture}(270,53)
\put( 10,25){$E_6$}
\put(132, 3){$\alpha_2$}
\put( 61,43){$\alpha_1$}
\put( 91,43){$\alpha_3$}
\put(121,43){$\alpha_4$}
\put(151,43){$\alpha_5$}
\put(181,43){$\alpha_6$}
\put(125, 5){\circle*{5}}
\put( 65,35){\circle*{5}}
\put( 95,35){\circle*{5}}
\put(125,35){\circle*{5}}
\put(155,35){\circle*{5}}
\put(185,35){\circle*{5}}
\put(125,35){\line(0,-1){30}}
\put( 65,35){\line(1,0){30}}
\put( 95,35){\line(1,0){30}}
\put(125,35){\line(1,0){30}}
\put(155,35){\line(1,0){30}}
\end{picture}
\end{center}
Again, since $p=2$, we do not have to worry about the precise choice of a 
Chevalley basis in the underlying Lie algebra. Further note that $G(\F_2)$ is
a simple group; it will just be denoted by $E_6(\F_2)$. (For $q$ an even 
power of~$2$, the group $G(\F_q)$ has a normal subgroup of index~$3$). Now 
we essentially rely on the knowledge of the character table of $E_6(\F_2)$, 
which has been determined by B.~Fischer and is contained in the {\sf GAP} 
library~\cite{bre}. (We have $|\Irr(E_6(\F_2))|=180$.) Although the group is
bigger, the further discussion will actually be simpler than that for type 
$F_4$ in the previous section, since there are fewer cuspidal unipotent 
character sheaves.

\begin{rema} \label{r61} By Lusztig \cite[20.3]{L2d} and Shoji 
\cite[4.6]{S3}, there are two cuspidal character sheaves $A_1,A_2\in
\hat{G}^{\text{un}}$. Both $A_1$, $A_2$ have the same support, namely,
the closure of the $G$-conjugacy class $C$ of an element $g_1=su\in G^F$ 
such that $s\in G^F$ is semisimple, with $C_G(s)^\circ$ isogneous to 
$\mbox{SL}_3(k)\times \mbox{SL}_3(k)\times\mbox{SL}_3(k)$, and $u \in 
C_G^\circ(s)^F$ is regular unipotent. Furthermore, $\dim C_G(g_1)=6$ and 
$A(g_1)\cong C_3\times C_3'$ where $C_3, C_3'$ are cyclic of order~$3$ 
and $C_3$ is generated by the image $\overline{g}_1\in A(g_1)$. Finally,
$A_j\leftrightarrow (g_1,\lambda_j)$ where $\lambda_j$ are linear characters 
which are trivial on $C_3'$ and such that $\lambda_1(\overline{g}_1)=
\theta$, $\lambda_2(\overline{g}_1)=\theta^2$ (where, again, $\theta$ is a
fixed third root of unity). Once we fix a specific element $g_1\in C^F$, we 
obtain well-defined characteristic functions 
\[\chi_j:=\chi_{g_1,\lambda_j}=\chi_{A_j,\phi_{A_j}} \qquad \mbox{for 
$j=1,2$ (cf.\ \ref{exp2})}.\]
\end{rema}

\begin{table}[htbp] \caption{Unipotent characters in the family $\cF_0$ for
type $E_6$, $q=2$} \label{uni2a}
\begin{center}
$\begin{array}{crc} \hline \rho_x & \rho_x(1) & \text{{\sf GAP}} \\\hline 
\{E_6[\theta],E_6[\theta^2]\}&   45532800& \{X.14,X.15\}   \\ 
D_4,r  &120645056 & X.18 \\
(20,10) & 184660800 & X.20 \\
(10,9) & 192047232 & X.22\\
(60,8) & 800196800 & X.29\\
(80,7) & 864212544 & X.30\\
(90,8) & 902358912 & X.31\\
\hline \multicolumn{3}{l}{\text{(Notation for $\rho_x$ as in Carter
\cite[p.~480]{C2})}} \end{array}$
\end{center}
\end{table}

\begin{rema} \label{r62} Let us write $\fU(G^F)=\{\rho_x\mid x\in 
X(W)\}$ and $\hat{G}^{\text{un}}=\{A_x \mid x \in X(W)\}$ as in \ref{r18}. 
Here, the parameter set $X(W)$ has $30$ elements, and it is partitioned into 
$17$ families. The $2$ elements of $X(W)$ which label cuspidal character 
sheaves are all contained in one family, which we denote by $\cF_0$ and 
which consists of $8$ elements in total. The elements of $\cF_0$ are given 
by the $8$ pairs $(u,\sigma)$ where $u\in\fS_3$ (up to $\fS_3$-conjugacy) and 
$\sigma\in\Irr(C_{\fS_3}(u))$. The diagonal block of the Fourier matrix
corresponding to $\cF_0$ is then described in terms of $\fS_3$ and the 
pairs $(u,\sigma)$; see \cite[\S 13.6]{C2} for further details. The 
identification of the $8$ unipotent characters $\{\rho_x\mid x\in\cF_0\}$ 
with characters in the {\sf GAP} table of $E_6(\F_2)$ is given in 
Table~\ref{uni2a}. (These characters are uniquely determined by their
degrees, except for the two cuspidal unipotent characters $E_6[\theta]$,
$E_6[\theta^2]$ which have the same degree.) The results of Shoji 
\cite[4.6, 5.2]{S3} show that the cuspidal character sheaves $A_j$ in 
\ref{r61} are labelled as follows by pairs in $\cF_0$:
\[ A_1=A_{(g_3,\theta)}\qquad \mbox{and} \qquad A_2=A_{(g_3,\theta^2)},\]
where we follow again the notation in \cite[p.~455]{C2} for pairs in 
$\cF_0$. Using the $8\times 8$-times Fourier matrix printed in 
\cite[p.~457]{C2} and the labelling of the unipotent characters $\{\rho_x 
\mid x \in \cF_0\}$ in terms of pairs $(u,\sigma)$ as above (see
\cite[p.~480]{C2}), we obtain the unipotent almost characters labelled 
by $(g_3,\theta)$ and $(g_3,\theta^2)$:
\begin{align*}
R_{(g_3,\theta)} & =\textstyle\frac{1}{3}\bigl(\rho_{(80,7)}+\rho_{(20,10)}
-\rho_{(10,9)}-\rho_{(90,8)}+2E_6[\theta]-E_6[\theta^2]\bigr),\\
R_{(g_3,\theta^2)} & =\textstyle\frac{1}{3}\bigl(\rho_{(80,7)}+\rho_{(20,10)}
-\rho_{(10,9)}-\rho_{(90,8)}-E_6[\theta]+2E_6[\theta^2]\bigr).
\end{align*}
\end{rema}

\begin{rema} \label{r62a} Let $C$ be the $F$-stable conjugacy class of $G$ 
such that $\mbox{supp}(A_j)=\overline{C}$ for $j=1,2$ (see \ref{r61}). 
We now fix a specific representative $g_1\in C^F$, as follows. Using the 
{\sf GAP} character table of $E_6(\F_2)$ and the information in 
Table~\ref{uni2a}, we find (for $q=2$) that $\{R_{(g_3,\theta)}, R_{(g_3, 
\theta^2)}\}=\{f_1,f_2\}$ where $f_1,f_2 \colon E_6(\F_2)\rightarrow 
\overline{\Q}_\ell$ are class functions with non-zero values as follows: 
\[ \begin{array}{cccc} \hline {\sf GAP:} & {\tt 12n} \,\text{(no.~60)} & 
{\tt 12o} \,\text{(no.~61)} & {\tt 12p} \, \text{(no.~62)} \\ \hline f_1 
& 8 & 8\theta & 8\theta^2\\f_2 & 8 & 8\theta^2 & 8\theta\\\hline\end{array}\]
Since $\chi_j$ only has non-zero values on elements in $C^F$, we conclude 
that $C$ must contain the elements of $E_6(\F_2)$ which belong to the three 
classes {\tt 12n}, {\tt 12o}, {\tt 12p} in the {\sf GAP} table; these three 
classes have centraliser order $192=3\cdot 2^6$. By inspection of the 
{\sf GAP} table, all character values on {\tt 12n} are integers, while the 
character values on {\tt 12o} are complex conjugates of those on {\tt 12p}. 
Thus, we choose $g_1\in C(\F_2)\subseteq C^F$ to be from the class {\tt 12n}
in $E_6(\F_2)$. Then $g_1$ is conjugate in $E_6(\F_2)$ to all powers
$g_1^n$ where $n$ is a positive integer coprime to the order of~$g_1$. 
Once $g_1$ is fixed, we obtain characteristic functions $\chi_1,\chi_2$ as
in \ref{r61}. Then Theorem~\ref{r16} shows that $R_{(g_3,\theta)}=
\zeta_{(g_3,\theta)} \chi_1$ and $R_{(g_3,\theta^2)}=\zeta_{(g_3,\theta^2)}
\chi_2$.
\end{rema}

\begin{rema} \label{r62b} Let $g_1\in C(\F_2)$ be as in \ref{r62a}. 
We write $g_1=su=us$ where $s\in E_6(\F_2)$ has order~$3$ and $u\in
E_6(\F_2)$ is unipotent. As in \ref{r52}(4) one sees that all such 
elements~$s$ of order~$3$ are conjugate in $E_6(\F_2)$. We now identify 
the unipotent part~$u$ of $g_1$, where we use the results and the notation
of Mizuno \cite{Miz}. Let 
\begin{equation*}
u_{15}:=x_{\alpha_2}(1)x_{\alpha_3}(1)x_{\alpha_4}(1)x_{\alpha_2+\alpha_4+
\alpha_5}(1)\in E_6(\F_2),\tag{a}
\end{equation*} 
as in \cite[Lemma~4.3]{Miz}; in particular, we have 
\begin{equation*}
A(u_{15})\cong \mathfrak{S}_3 \qquad \mbox{and} \qquad |C_G(u_{15})^F|=
6q^{18}(q-1)^2.\tag{b}
\end{equation*} 
We claim that the unipotent part $u$ of $g_1$ is conjugate in $E_6(\F_2)$
to $u_{15}$. To see this, we follow a similar approach as in \ref{r52}. 
Using the explicit $78$-dimensional matrix realization of $E_6(\F_2)$ from 
Lusztig \cite[2.3]{L12} (see also \cite[4.10]{mylie}), we create $E_6(\F_2)$ 
as a matrix group in {\sf GAP}. We consider the subgroup 
\[L:=\langle x_{\pm \alpha_2}(1),x_{\pm \alpha_3}(1),x_{\pm \alpha_4}(1), 
x_{\pm \alpha_5}(1) \rangle \subseteq E_6(\F_2);\]
note that $u_{15}\in L$. The group $L$ is simple of order $174182400$. (In 
fact, $L \cong  \mbox{SO}_8^+(\F_2)$.) As in Example~\ref{so8}, the 
{\sf GAP} function {\tt CharacterTable} computes the character table of $L$ 
(by a general algorithm), together with a list of representatives of the 
conjugacy classes; there are $53$ conjugacy classes of $L$. One notices that 
there is a unique class of $L$ in which the elements have centraliser 
order~$64$, and one checks using {\sf GAP} that our element~$u_{15}$ belongs 
to this class. Now, as in \ref{r52}, the function {\tt PossibleClassFusions} 
determines all possible fusions of the conjugacy classes of $L$ into the 
{\sf GAP} table of $E_6(\F_2)$. It turns out that there are $2$ possible 
fusion maps. But under each of these two possibilities, the class 
containing our element~$u_{15}$ is mapped to the class {\tt 4k} (no.~18) 
in the {\sf GAP} table of $E_6(\F_2)$. Thus, $u_{15}$ belongs to the class 
{\tt 4k} in {\sf GAP}. According to the {\sf GAP} table, the element 
$g_1^3=u^{-1}$ also belongs to the class {\tt 4k}; furthermore, all 
character values on {\tt 4k} are integers. Hence, $u$, $u^{-1}$, $u_{15}$ 
are all conjugate in $E_6(\F_2)$, as claimed. 
\end{rema}

To summarize the above discussion, we can assume that the chosen 
representative $g_1\in C^F$ in \ref{r62a} is as follows. We have $g_1=
us=su\in E_6(\F_2)$ where $u=u_{15}$ (see \ref{r62b}(a)) and $s\in 
E_6(\F_2)$ is any element of order~$3$ such that $su=us$.

\begin{prop} \label{r63} Let $g_1=su_{15}\in E_6(\F_2)\subseteq G^F$ be as
above. For $j=1,2$, define $\chi_j=\chi_{(g_1,\lambda_j)} \colon G^F
\rightarrow \overline{\Q}_\ell$ as in Example~\ref{exp2}, with $\lambda_j$ 
as in \ref{r61}. Then $\zeta_{(g_3, \theta)}=\zeta_{(g_3,\theta^2)}=1$;
hence, $R_{(g_3,\theta)}=\chi_1$ and $R_{(g_3,\theta^2)}=\chi_2$. 
\end{prop}

\begin{proof} Recall from \ref{r62} that $R_{(g_3,\theta)}=\zeta_{(g_3,
\theta)}\chi_1$ and $R_{(g_3,\theta^2)}=\zeta_{(g_3,\theta^2)} \chi_2$. 
Now it will be sufficient to show that the scalars $\zeta_{(g_3,\theta)}$ 
and $\zeta_{(g_3,\theta^2)}$ are equal to~$1$ for the special case where 
$q=2$; see Proposition~\ref{p1}. But this has been observed in~\ref{r62a} 
above; note that, for $q=2$, we have $\{f_1,f_2\}=\{\chi_1,\chi_2\}$ and
$g_1$ belongs to the class {\tt 12n} in the {\sf GAP} table. Regardless of 
how we match the functions in these two pairs, the values on $g_1\in 
{\tt 12n}$ are equal to~$8$ in both cases. Hence, we must 
have $\zeta_{(g_3, \theta)}=\zeta_{(g_3,\theta^2)}=1$.
\end{proof}

\begin{rem} \label{fine6} We briefly sketch an alternative proof of 
Proposition~\ref{r63}, following the line of argument in Lusztig 
\cite{L3}. Let us fix $g_1=su_{15}=u_{15}s$ as above. Arguing analogously 
to \ref{r62}, we can express the unipotent characters $\{\rho_x\mid
x\in \cF_0\}$ in terms of the almost characters $\{R_x\mid x\in \cF_0\}$.
In particular, we obtain:
\[E_6[\theta] =\textstyle\frac{1}{3}\bigl(R_{(80,7)}+R_{(20,10)}
-R_{(10,9)}-R_{(90,8)}+2R_{(g_3,\theta)}-R_{(g_3,\theta^2)}\bigr).\]
Now evaluate on $g_1$. The values of the almost characters $R_\epsilon$ 
($\epsilon\in\Irr(W)$) occurring in the above expression can be computed 
using the character formula in \cite[7.2.8]{C2} (which expresses the values
of the Deligne--Lusztig virtual characters $R_w$ in terms of the Green 
functions of $C_G(s)^\circ$) and the fact that Green functions take value~$1$
on regular unipotent elements (see \cite[9.16]{DeLu}). Furthermore, since all 
character values on $g_1$ are integers, we must have $\zeta_{(g_3,\theta)}=
\zeta_{(g_3, \theta^2)}=\pm 1$, and so $R_{(g_3,\theta)}(g_1)=
R_{(g_3,\theta^2)}(g_1)=\pm q^3$. This yields the condition that  
\[ E_6[\theta](g_1)=\textstyle\frac{1}{3}\bigl(\pm q^3\,+\text{ known 
value}\bigr)\in \Z,\]
which uniquely determines the sign of $\zeta_{(g_3,\theta)}$. Note
that this argument does not require the knowledge of the complete character
table of $E_6(\F_2)$, but we still need to know some information about the 
representative~$g_1$, e.g., the fact that it is conjugate to all powers 
$g_1^n$ where~$n$ is coprime to the order to~$g_1$. (A similar argument will 
actually work for $E_6(\F_q)$, where $q$ is a power of any prime $\neq 3$, 
once the existence of a representative $g_1$ as above is guaranteed; for $q$ 
a power of~$3$, a different argument is needed. See \cite{GeHe} for further 
details.)
\end{rem}

\medskip
\noindent
{\bf Acknowledgements.} We are indebted to George Lusztig and Toshiaki 
Shoji for useful comments concerning the proof of Proposition~\ref{p1}. 
Thanks are due to Frank L\"ubeck for pointing out Marcelo--Shinoda 
\cite{MaSh}, and to Thomas Breuer who provided efficient help with the various 
character table functions in {\sf GAP}; he also made available the character
table of $E_6(\F_2)$. We thank Gunter Malle and Lacri Iancu for a careful 
reading of the manuscript and critical comments. Finally, this work was 
supported by DFG SFB-TRR 195.


\end{document}